\documentclass{article}
\usepackage{amsmath}
\usepackage{amssymb}
\usepackage{amsthm}
\RequirePackage{authblk}

\usepackage{mathrsfs}

\newtheorem{lemma}{Lemma}
\newtheorem{corollary}[lemma]{Corollary}
\newtheorem{proposition}[lemma]{Proposition}
\newtheorem{remark}[lemma]{Remark}

\newtheorem{theorem}[lemma]{Theorem}

\usepackage{graphicx} 
\usepackage{caption}
\usepackage{subcaption}
\usepackage{pgfplots,tikz}
\usetikzlibrary{arrows,calc,decorations.markings,math,arrows.meta,intersections,svg.path}
\newcommand{\re}{\mathbb R}
\newcommand{\e}{\varepsilon} 
\newcommand{\f}{\varphi}
\newcommand{\dd}{N}
\newcommand{\sphr}[1]{{#1}^{\#}}
\newcommand{\wphi}{w}
\newcommand{\Rbar}{R_0}

\newcommand{\msF}{\mathscr F}

\usepackage[pdftex,hyperindex]{hyperref}
\hypersetup{
  bookmarks=false,
	colorlinks=true,
	urlcolor=blue,
}
\title{
Rearrangements and the Maz'ya–Shaposhnikova limit: $\Gamma$-convergence for truncated interactions}
\author[1]{Andrea Braides} 
\author[2]{Anneliese Defranceschi}
\author[3]{Enrico Vitali} 
\affil[1]{\small Department of Mathematics, University of Rome Tor Vergata, via~della Ricerca~Scientifica 1, 00133 Rome, Italy {\tt braides@mat.uniroma2.it}}
\affil[2]{\small Department of Mathematics, University of Trento, via Sommarive 14, 38123~Povo~(Trento), Italy {\tt anneliese.defranceschi@unitn.it}}
\affil[3]{\small Department of Mathematics, University of Pavia, via A. Ferrata 5, Pavia, Italy
{\tt enrico.vitali@unipv.it}}
\date{}
\begin{document}
\maketitle
\begin{abstract}
We give a proof of the $\Gamma$-liminf inequality for the Maz'ya--Shaposhni\-kova limit, as
$s\to0^+$, of the fractional Sobolev seminorm, based on spherical rearrangements. The key
point is that the rearrangements of a weakly convergent sequence are monotone radial
profiles, which converge strongly by Helly's theorem; since the energy does not increase under
rearrangement (in the limit), this reduces the liminf inequality to a pointwise convergence
statement for radially decreasing functions.
As an application, we treat truncated energies, in which the interaction is restricted to a ball
of radius $r_s\to\infty$. Assuming $r_s^{-2s}\to\lambda\in[0,1]$, we prove that the scaled
energies $\Gamma$-converge, with respect to weak $L^2$-convergence, to
$$
(1-\lambda)N\omega_N\|u\|_{L^2}^2,
$$
where $\omega_N$ is the measure of the unit ball of $\mathbb R^N$. This identifies the
contribution of the long-range interactions and gives a continuous interpolation between the
classical Maz'ya--Shaposhnikova limit ($\lambda=0$) and the vanishing-energy regime
($\lambda=1$).
\smallskip
{\small \par\noindent

}
\end{abstract}
\bigskip

\section{Introduction}
The asymptotic behaviour of fractional Sobolev seminorms as the fractional exponent approaches the endpoints $0$ and $1$ has attracted considerable attention during the last two decades. While the limit $s\to1^{-}$ is described by the Bourgain--Brezis--Mironescu formula \cite{BBM}, the opposite regime was characterized by Maz'ya and Shaposhnikova \cite{Maz-Shap02}. In particular, for every $1< p<\infty$, the appropriately normalized $p$-th power of the fractional $p$-seminorm converges, as $s\to0^+$, to a multiple of the $L^p$-norm, provided that $u$ belongs to $W^{s,p}(\mathbb R^N)$ for some $s\in(0,1)$. In the case $p=2$ considered here, this gives
$$
\lim_{s\to0} s
\iint_{\mathbb R^N\times\mathbb R^N}
\frac{|u(x)-u(y)|^2}{|x-y|^{N+2s}}
\,dx\,dy
=
N\omega_N\|u\|_{L^2(\mathbb R^N)}^2,
$$
where $\omega_N$ denotes the measure of the $N$-dimensional unit ball.

Several extensions of the Maz'ya--Shaposhnikova formula have been obtained subsequently, including results for fractional Orlicz--Sobolev spaces \cite{AlbericoCianchiPickSlavikova20}, magnetic fractional Sobolev spaces \cite{PinamontiSquassinaVecchi17}, and more general geometric and interaction settings \cite{HanPinamonti21,HanPinamontiXuZambanini25,DavoliDiFrattaGiorgioPinamonti25}. These results show that the Maz’ya--Shaposhnikova limit is sensitive both to the geometry of the underlying space and to the behaviour of the interaction at large distances.

Besides its intrinsic interest, the Maz'ya--Shaposhnikova formula has also become a fundamental tool in the analysis of nonlocal variational problems. In particular, Crismale, De Luca, Kubin, Ninno and Ponsiglione \cite{PoDeKuNiCr21} showed that the classical limit admits a variational interpretation in terms of $\Gamma$-convergence with respect to weak $L^2(\Omega)$-convergence for functions extended to $0$ outside the bounded set $\Omega\subset\mathbb R^N$. More recently, Davoli, Fanizza and Picerni \cite{DavFanPic26} obtained a generalized Maz'ya--Shaposhnikova formula for functions that need not belong to $L^2(\mathbb R^N)$, showing that the limit is corrected by a suitable notion of mass at infinity and establishing a close connection with the asymptotic behaviour of fractional perimeters.

In this paper we propose a different route to the $\Gamma$-liminf inequality for the
Maz'ya--Shaposhnikova limit, based on spherical rearrangements. The difficulty is that the
energies are not weakly continuous in $L^2$, so some compactness mechanism is needed to pass
from a weakly convergent sequence $(u_s)$ to its limit. Here this is provided by
rearrangement: the symmetric decreasing rearrangements $u_s^\#$ are monotone radial profiles,
so by Helly's theorem they converge strongly, and the energy does not increase under
rearrangement (in the limit) by the P\'olya--Szeg\H{o} inequality. Since $u_s^\#$ and $u_s$ have
the same $L^2$ norm, the liminf inequality is thus reduced to the pointwise convergence of the
energy on radially decreasing functions, which we prove by approximation with nested step
functions. The argument is qualitative: it uses a rearrangement inequality and compactness of
monotone profiles, and no explicit integration of the kernel.

As an application we study how the Maz'ya--Shaposhnikova limit is affected when the interaction
is truncated. We consider an interaction radius $r_s\to+\infty$ and assume that
$r_s^{-2s}\to\lambda\in[0,1]$. Then the scaled energies $\Gamma$-converge, with respect to the
weak $L^2(\Omega)$-topology (for functions extended to $0$ outside a bounded set $\Omega$), to
$$
(1-\lambda)N\omega_N\int_\Omega |u(x)|^2\,dx.
$$
The parameter $\lambda$ measures the asymptotic amount of interaction that is `lost at
infinity'. When $r_s\equiv+\infty$ one has $\lambda=0$ and the classical formula is recovered.
At the opposite extreme, if $\log r_s=o(1/s)$, then $\lambda=1$ and the limiting energy
vanishes; every intermediate regime ($\log r_s\sim\frac{1}{2s}|\log \lambda|$) produces a
continuous interpolation between these two situations.
The upper bound follows from the classical Maz'ya--Shaposhnikova formula and approximation
arguments. For sequences supported in a fixed ball the lower bound can also be obtained by
integrating the kernel directly over the region where one of the two points lies outside the
support; we regard the rearrangement argument as interesting for its structure, since it
handles the lack of weak continuity by a compactness property rather than by an explicit
computation.

From a broader perspective, the present result complements \cite{DavFanPic26}. While the
latter shows how the limit changes when the admissible functions possess a nontrivial mass at
infinity, here we show how it changes when the interaction with infinity is progressively
removed. The two phenomena describe different mechanisms through which the behaviour at
infinity governs the asymptotics of fractional Sobolev energies.

We conclude with some possible variations and extensions of our asymptotic result.
The statements should carry over to the $p$-th power of the Gagliardo seminorm, $1<p<\infty$, with the weak $L^p$-topology.
In this case, the assumption on the radii becomes $r_s^{-sp}\to\lambda\in[0,1]$, and the limit energy is
$(1-\lambda)\tfrac{2}{p}N\omega_N\|u\|_{L^p}^p$, in accordance with \cite{Maz-Shap02}.
Indeed, the decomposition in Remark~\ref{rk:dec_F}, the P\'olya--Szeg\H{o} inequality, and Helly's theorem
do not rely on the Hilbert structure.
Moreover, since our argument is based on the monotonicity of the energy under spherical rearrangement,
it is natural to expect that it applies to radially symmetric decreasing kernels in the spirit of \cite{DavoliDiFrattaGiorgioPinamonti25}.
For kernels without these properties, a different approach to the $\Gamma$-liminf inequality would be needed.

An alternative to restricting interactions to $B_{r_s}\times B_{r_s}$, which cannot be written as a convolution-type kernel, is to truncate the kernel itself, replacing it with $\mathbf 1_{\{|x-y|<r_s\}}|x-y|^{-N-2s}$.
For compactly supported functions, we expect that the same value of $\lambda$ governs the limit; whether this fits the framework of \cite{HanPinamonti21} is a natural question.
It would also be interesting to study the corresponding truncated $s$-perimeters, in connection with \cite{DFPV}, and to combine the truncation with functions having nontrivial mass at infinity, as in \cite{DavFanPic26}.
In both cases, the interplay between the radius of the truncation $r_s$ and the behaviour at infinity is the main issue.

\smallskip
The paper is organized as follows. In Section \ref{sec:main} we state the main results. In Section \ref{sec:convergence} we 
establish a pointwise convergence result for spherically symmetric radially decreasing functions with a $C^1$ profile. In Section \ref{sec:rearr} we give details of the symmetrization argument to reduce the proof of the liminf inequality to the pointwise convergence proved in Section \ref{sec:convergence}, concluding the proof of Theorem \ref{thm:main}. Section \ref{sec:limsup} gives the pointwise convergence {for Lipschitz functions with compact support} and completes the proof of the recovery sequence; that is, Theorem \ref{thm:pws}.

\section{Main results}\label{sec:main}
Let $u\colon\re^\dd\to\re$ be a Lebesgue-measurable function. Let $s\in(0,1)$ and define the Gagliardo $s$-seminorm of a function $u$ on the set $A$ as
$$
|u|_{2,s}(A)=\left(\iint_{A\times A}\frac{|u(x)-u(y)|^2}{|x-y|^{\dd+2s}}\, dx\, dy\right)^{1/2}.
$$
If $A=\re^\dd$ we simply write $|u|_{2,s}$.

We fix a family of radii $r_s\to+\infty$  as $s\to 0$, with 
\begin{equation}\label{lambda}
\lim_{s\to0}r_s^{-2s}=\lambda\in[0,1].    
\end{equation}
This condition is equivalent to $\log r_s>\!>1/s$ if $\lambda=0$, $\log r_s\sim |\log\lambda|/2s$ if $\lambda\in(0,1)$, and $\log r_s<\!<1/s$ if $\lambda=1$.

Using the notation $B_{r}=B_{r}(0)$ for the ball centered at $0$ and with radius $r$,
we consider the truncated energies
\begin{equation}\label{sGs}
\msF_s(u):= s|u|^2_{2,s}(B_{r_s})=s\iint_{B_{r_s}\times B_{r_s}}\frac{|u(x)-u(y)|^2}{|x-y|^{\dd+2s}}\, dx\, dy.
\end{equation}
We allow $r_s\equiv +\infty$, in which case $\lambda=0$ and for which $B_{r_s}$ is replaced by $\re^\dd$. 
We introduce a separate notation for this case and define $\msF^0_s(u)=s|u|^2_{2,s}(\re^\dd)$.

\smallskip
\smallskip
The core of the paper is the following liminf theorem, which is proved by spherical
rearrangement in Sections~\ref{sec:convergence} and~\ref{sec:rearr}. The case $\lambda=0$ is
the $\Gamma$-liminf inequality for the classical Maz'ya--Shaposhnikova limit (for
$r_s\equiv+\infty$, and more generally for $r_s$ growing so fast that $r_s^{-2s}\to0$), so that
the theorem contains a new proof of that inequality; the truncation enters through the
parameter $\lambda$.

\begin{theorem}\label{thm:main}
Let $\msF_s$ be given by \eqref{sGs} and let $\lambda$ be defined by \eqref{lambda},
let $(u_n)$ be a sequence in $L^2(\re^\dd)$
 weakly converging to a function $u$, and let $(s_n)$ be an infinitesimal sequence in $(0,1)$.
 Assume that the functions $u_n$ vanish outside a common compact set. 
 Then
 $$
  \liminf_{n\to+\infty} \msF_{s_n}(u_n)\ge
  (1-\lambda)\dd \omega_\dd \liminf_{n\to+\infty}\|u_n\|_{L^2(\re^\dd)}^2.
 $$
\end{theorem}

\begin{remark}\label{rk:direct}\normalfont
When the functions $u_n$ vanish outside a common ball $B_{R_0}$, the inequality in
Theorem~\ref{thm:main} can also be obtained by integrating the kernel directly. If $x\in B_{R_0}$
and $2R_0<|z|<r_s-R_0$, then $x+z\in B_{r_s}\setminus B_{R_0}$, where $u_n$ vanishes, so that
\begin{eqnarray*}
\msF_{s}(u_n)&\ge& 2s\int_{B_{R_0}}|u_n(x)|^2\int_{\{2R_0<|z|<r_s-R_0\}}\frac{dz}{|z|^{\dd+2s}}\,dx
\\
&&= \dd\omega_\dd\Big[(2R_0)^{-2s}-(r_s-R_0)^{-2s}\Big]\|u_n\|_{L^2}^2,
\end{eqnarray*}
and the bracket tends to $1-\lambda$. The argument below is different in nature: it is
qualitative, and it handles the lack of weak continuity of the energies by a compactness
property of monotone radial profiles (the rearrangements converge strongly), rather than by an
explicit computation with the kernel.
\end{remark}

This result is complemented by the next theorem, which is proved following an argument in \cite[Section 4]{Maz-Shap02}.

\begin{theorem}\label{thm:pws}
 Let 
 $u$ be a Lipschitz function on $\re^\dd$  with compact support. 
 Then
 $$
 \limsup_{s\to0} \msF_s(u)\le  (1-\lambda)\dd \omega_\dd \|u\|^2_{L^2(\re^\dd)}.
 $$
\end{theorem}

Note that, in particular, the two results above imply the convergence of $\msF_s(u)$ to 
$(1-\lambda)\dd\omega_\dd\|u\|^2_{L^2}$ for every 
$u\in C^1_c(\re^\dd)$.

\begin{remark}\label{rk:dec_F}\normalfont
It is useful to point out the special role of the case $r_s\equiv+\infty$. Let $u$ be a function with support in a ball $B_{R_0}$, and let $r_s>R_0$. 
The domain of integration $\re^\dd\times\re^\dd$ in the definition of $\msF_s^0$ can be split taking into account the disappearance of $u$ outside $B_{r_s}$. Thus, we can decompose $\msF_s(u)$ as follows:
\begin{equation}\label{eq:dec_F}
\begin{split}
\msF_s(u)=&\ \msF_s^0(u)-\delta_s(u^2),\\
&\text{with}\ \ 
\delta_s(f)=2s\iint_{B_{R_0}\times (\re^\dd\setminus B_{r_s})}\frac{f(x)}{|x-y|^{\dd+2s}}\, dx\, dy.
\end{split}
\end{equation}
This decomposition will be applied appropriately in the sequel to recover the general case from the asymptotics of $\msF_s^0$. For example, Proposition \ref{prop:corr} in Section~\ref{sec:convergence} implies that if $u\in C^1_c(\re^\dd)$ then $\delta_s(u^2)$ tends to $\lambda\dd\omega_\dd \|u\|^2_{L^2}$ as $s\to0$. As a consequence, it is enough to prove Theorem \ref{thm:pws} for $\msF_s^0$.
%
Note that Theorem~\ref{thm:main} cannot be derived from the case $r_s\equiv+\infty$ through this
decomposition, since the convergence of $\delta_s(u^2)$ fails if $u$ is replaced by a weakly
convergent sequence; this is one reason why a compactness mechanism, here provided by
rearrangements, is needed.
\end{remark}

 We can derive a $\Gamma$-convergence result as a corollary of the previous theorems. For
$\lambda=0$ it recovers the $\Gamma$-convergence of \cite{PoDeKuNiCr21}, and for $\lambda>0$ it
describes the effect of truncation.
Let $\Omega$ be a bounded open subset of $\re^\dd$ with Lipschitz boundary. For every $u\in L^2(\Omega)$ we denote by $\tilde u\colon\re^\dd\to\re$ the extension of $u$ to the whole $\re^\dd$ with value zero outside $\Omega$.
For every $s\in (0,1)$ we define
$$
F_s\colon L^2(\Omega)\to [0,+\infty],\qquad F_s(u)=\msF_s(\tilde u),
$$
and $F_0(u)=(1-\lambda)\dd \omega_\dd\|u\|_{L^2(\Omega)}^2$.

\begin{corollary}\label{cor:Gamma}
Let $(s_n)$ be an infinitesimal sequence in $(0,1)$. Then the functionals $F_{s_n}$ $\Gamma$-converge to $F_0$ with respect to the weak $L^2(\Omega)$-convergence  as $n\to+\infty$.
\end{corollary}

\proof ($\Gamma$-$\liminf$) Let $(u_n)$ be a sequence in $L^2(\Omega)$
 weakly converging to a function $u$. Let $\tilde u_n$ and $\tilde u$ denote the extension of $u_n$ and $u$, respectively, to the whole $\re^\dd$ with zero value outside $\Omega$.
Since $\Omega$ is bounded, and $(\tilde u_n)$ weakly converges to $\tilde u$ in $L^2(\re^\dd)$, we can apply Theorem \ref{thm:main}, and deduce that
\begin{align*}
\liminf_{n\to+\infty} F_{s_n}(u_n)&=\liminf_{n\to+\infty} \msF_{s_n}(\tilde u_n)\\
    &\ge (1-\lambda)\dd\omega_\dd\|\tilde u\|_{L^2(\re^\dd)}^2
    =(1-\lambda)\dd\omega_\dd\|u\|_{L^2(\Omega)}^2\,.
\end{align*}
\par\noindent
($\Gamma$-$\limsup$) Theorem \ref{thm:pws} guarantees the $\Gamma$-$\limsup$ inequality for $C^\infty_c(\Omega)$  functions. The general case follows by a standard approximation and diagonalization argument.
 \qed
\medskip

In the next two sections we prove Theorem~\ref{thm:main} by rearrangements: Section~\ref{sec:convergence}
establishes the pointwise convergence on radially decreasing functions, and Section~\ref{sec:rearr}
reduces the general case to it. Theorem~\ref{thm:pws} is proved in Section~\ref{sec:limsup}.

\par
In order to simplify the notation, the sequences $(u_n)_n$ in the previous statements will be denoted by $(u_s)_s$ (instead of $(u_{s_n})_n$), where $s$ is intended to vary
along a fixed infinitesimal positive sequence. 

\section{Convergence on radially decreasing functions}\label{sec:convergence}

We first prove a pointwise convergence result for {spherically symmetric radially decreasing functions with a $C^1$ profile and compact support.}
In the next section we reduce to this case by a rearrangement argument.

\begin{proposition}\label{prop:pws_radial}
 Let $\phi\colon[0,+\infty)\to[0,M]$ be a $C^1$ decreasing function, with $\phi(r)=0$ if $r\ge R_0$ for some $R_0>0$. 
 Let $\wphi(x)=\phi(|x|)$. Then
 \begin{equation}\label{eq:ptw_conv}
   \lim_{s\to 0} \msF_s(w)=(1-\lambda)\dd\omega_\dd\|\wphi\|_{L^2(\re^\dd)}^2.
 \end{equation}
\end{proposition}

We preliminarily prove a result about the kernel in the definition of $\delta_s$ (see~\eqref{eq:dec_F}) needed in the proof of Proposition \ref{prop:pws_radial}.

\begin{proposition}\label{prop:corr}
For any $x\in \re^\dd$, define
$$
\eta_s(x)=2s\int_{\re^\dd\setminus B_{r_s}}\frac1{|x-y|^{\dd+2s}}\,dy.
$$
Then $\eta_s\to \lambda\dd\omega_\dd$ uniformly on compact sets.
\end{proposition}

\proof  
We set 
$$
I_s=2s\int_{\re^\dd\setminus B_{r_s}}\frac1{|y|^{\dd+2s}}\,dy\,.
$$
If $|x|\le R<r_s$ then we have
$$
|\eta_s(x)- I_s|\le
    2s\int_{\re^\dd\setminus B_{r_s}}\frac1{|y|^{\dd+2s}}\left|\left(\frac{|y|}{|x-y|}\right)^{\dd+2s}-1\right| dy\,.
$$
For every $y\in\re^\dd\setminus B_{r_s}$, from the inequality $\bigl||x-y|-|y|\bigr|\le |x|\le R$ we deduce that
$|x-y|\ge r_s-R$, and
$$
\left|\frac{|y|}{|x-y|}-1\right|\le\frac{|x|}{|x-y|}\le\frac{R}{r_s-R}.
$$
Therefore, to estimate the integral above, for $r_s$ sufficiently large we consider the function 
$t\mapsto t^{\dd+2s}$ in a neighbourhood of $t=1$, where the function is Lipschitz.
It follows that there exists a constant $L>0$ such that for every $s$
$$
\bigg|\Big(\frac{|y|}{|x-y|}\Big)^{\dd+2s}-1\bigg|\le L \left|\frac{|y|}{|x-y|}-1\right|
    \le L \frac{R}{r_s-R}.
$$
Hence,
$$
|\eta_s(x)- I_s|\le L \frac{R}{r_s-R} I_s\,.
$$
A direct computation shows that $I_s$ tends to $\lambda \dd\omega_\dd$ as $s\to 0$, thus concluding the proof. \qed

\proof[Proof of Proposition {\rm\ref{prop:pws_radial}}.]
It is enough to consider the case $r_s\equiv +\infty$; i.e., $\msF_s^0$. Indeed, by the decomposition in \eqref{eq:dec_F}, we have $\msF_s(w)=\msF_s^0(w)-\delta_s(w^2)$, and $\delta_s(w^2)\to\lambda\dd\omega_\dd\|\wphi\|_{L^2}^2$ as $s\to0$
by Proposition \ref{prop:corr}.

The proof will proceed by approximation of $\wphi$ by radially symmetric step functions.
\par
Fix $h>0$ and let $n=\lfloor M/h \rfloor$. For every $i=0,\ldots,n$, let 
$$
B_i=\{x\in\re^\dd:\ \wphi(x)>i h\}=\{x\in \re^\dd:\ \phi(|x|)>ih\}.
$$
Note that the radial symmetry of $\wphi$ implies that $B_i$ is a ball centered at the origin. We denote by $r_i$ the radius. We can assume that $r_0=R_0$ and $B_n\neq\emptyset$.
Thus, since $\phi(r_i)=ih$ and $\phi$ is a $C^1$-function, we have
$0<r_{i+1}< r_i\le R_0$. 
\par
Define
\begin{equation}\label{eq:wh}
\wphi^h=\sum_{i=0}^n h 1_{B_i}.
\end{equation}

In Step 1 we prove equality \eqref{eq:ptw_conv} for the characteristic function of a ball; then (in Step 2 and Step 3)
we consider the general case of simple functions. In Step~4 we show that $\msF_s^0(\wphi)$ can be approximated
by $\msF_s^0(\wphi^h)$ as $h\to0$, \emph{uniformly with respect to $s$}. This allows us to conclude. Indeed, 
from Step 3 and 4 we deduce that
\begin{align*}
\lim_{s\to0} \msF_s^0(w)&=\lim_{s\to0}\lim_{h\to0}\msF_s^0(w^h)=\lim_{h\to0}\lim_{s\to0} \msF_s^0(w^h)\\
    &=\dd\omega_\dd \lim_{h\to0} \|w^h\|_{L^2(\re^\dd)}^2
    =\dd\omega_\dd\|w\|_{L^2(\re^\dd)}^2.
\end{align*}

\subparagraph{Step 1} We first prove \eqref{eq:ptw_conv} in the case where $\wphi=1_{B_r(0)}$, in which
\begin{equation}\label{st-1}
\lim_{s\to0} \msF_s^0(\wphi)=
\dd\omega_\dd\|\wphi\|_{L^2(\re^\dd)}^2.
\end{equation}

To show this, we note that, setting 
$$
\sigma(x)=\int_{B_r(0)\setminus B_{r-|x|}(x)}\frac1{|x-y|^{\dd+2s}}\, dy,
$$
we have
\begin{align*}
\msF_s^0(\wphi)&=s\int\int_{\re^\dd\times\re^\dd}\frac{|\wphi(x)-\wphi(y)|^2}{|x-y|^{\dd+2s}}\, dx dy\cr
  &=2s\int_{B_r(0)}\, dx \int_{\re^\dd\setminus B_r(0)}\frac1{|x-y|^{\dd+2s}}\,dy\cr
  &=2s\int_{B_r(0)}\left(\int_{\re^\dd\setminus B_{r-|x|}(x)}\frac1{|x-y|^{\dd+2s}}\, dy -\sigma(x)\right)\,dx.
\end{align*}

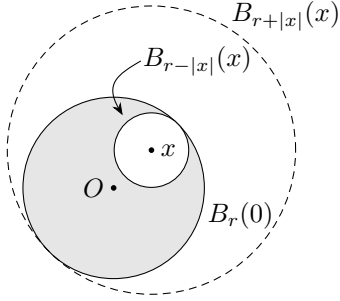
\begin{figure}[h!]
\begin{center}
\begin{tikzpicture}[scale=1.2]
 \draw [fill=gray!20] (0,0) circle [radius=1];
 \draw [fill=white] (5/12,5/12) circle [radius=1-0.5892];
 \draw [densely dashed] (5/12,5/12) circle [radius=1+0.5892];
 \draw [fill] (5/12,5/12) circle [radius=0.025];
 \node [right] at (5/12,5/12) {$x$};
 \draw [fill] (0,0) circle [radius=0.025];
 \node [left] at (0,0) {$O$};
 \node at (1.9,1.9) {$B_{r+|x|}(x)$};
  \node at (1.4,-.3) {$B_r(0)$};
\draw[{Stealth[scale length=1]}-] (0.1,.8) .. controls +(120:0.4cm) and +(-150:0.3cm) .. (0.3,1.4);
\node[right] at (0.22,1.4) {$B_{r-|x|}(x)$};
\end{tikzpicture}
\caption{in the estimate of $\sigma(x)$, the domain of integration (shaded part) is controlled by $B_{r+|x|}(x)\setminus B_{r-|x|}(x)$.}
\label{fig:A}
\end{center}
\end{figure}
Since $B_r(0)\subseteq B_{r+|x|}(x)$ (see Figure \ref{fig:A}), we get
\begin{align*}
\sigma(x)&\le \int_{B_{r+|x|}(x)\setminus B_{r-|x|}(x)}\frac1{|x-y|^{\dd+2s}}\, dy
 =\int_{B_{r+|x|}(0)\setminus B_{r-|x|}(0)}\frac1{|\xi|^{\dd+2s}}\, d\xi\cr
 &=\dd\omega_\dd\int_{r-|x|}^{r+|x|}\frac1{\varrho^{\dd+2s}}\varrho^{\dd-1}\,d\varrho
 =\frac{\dd\omega_\dd}{2s}\left[\frac1{(r-|x|)^{2s}}-\frac1{(r+|x|)^{2s}}\right]\,.
\end{align*}
Hence,
$$
2s\int_{B_r(0)} \sigma(x)\,dx\le (\dd\omega_\dd)^2
    \int_0^r\left[\frac1{(r-\varrho)^{2s}}-\frac1{(r+\varrho)^{2s}}\right]\varrho^{\dd-1}\,d\varrho\,.
$$
The integral on the right-hand side is majorized by
$$
r^{\dd-1}\int_0^r\left[\frac1{(r-\varrho)^{2s}}-\frac1{(r+\varrho)^{2s}}\right]\,d\varrho
    =\frac{r^{\dd-1}}{1-2s}[2r^{1-2s}-(2r)^{1-2s}]\underset{s\to0}\to0\,.
$$
Therefore,
$$
\msF_s^0(\wphi)=2s\int_{B_r(0)}dx\int_{\re^\dd\setminus B_{r-|x|}(x)}\frac1{|x-y|^{\dd+2s}}\, dy + o_s(1).
$$
A straightforward computation of the inner integral yields $\frac{\dd\omega_\dd}{2s}(r-|x|)^{-2s}$.
Therefore,
\begin{align*}
\msF_s^0(\wphi)=\dd\omega_\dd\int_{B_r(0)}(r-|x|)^{-2s}\, dx + o_s(1)=(\dd\omega_\dd)^2\int_0^r\frac{\varrho^{\dd-1}}{(r-\varrho)^{2s}}\,d\varrho+o_s(1).
\end{align*}
By applying Lebesgue's Dominated Convergence Theorem, the limit of the right-hand side as $s\to0$ gives
$$
(\dd\omega_\dd)^2\int_0^r \varrho^{\dd-1}\,d\varrho
    =\dd(\omega_\dd)^2 r^\dd=\dd\omega_\dd\|\wphi\|_{L^2}^2\,,
$$
which proves \eqref{st-1}.

\subparagraph{Step 2} Let $0<r<R$ and $h,H>0$. Then $\wphi=H 1_{{B_R(0)}}+h1_{{B_r(0)}}$
satisfies~\eqref{eq:ptw_conv}.

We first note that we have
\begin{align*}
|\wphi(x)-\wphi(y)|^2=H^2|&1_{B_R(0)}(x)-1_{B_R(0)}(y)|^2 + h^2|1_{B_r(0)}(x)-1_{B_r(0)}(y)|^2\cr
        &+ 2Hh\bigl(1_{B_R(0)}(x)-1_{B_R(0)}(y)\bigr)\bigl(1_{B_r(0)}(x)-1_{B_r(0)}(y)\bigr).
\end{align*}
If the last term does not vanish, then one of $x$ and $y$ must lie outside $B_R(0)$.
The other must belong to $B_r(0)$, since otherwise $1_{B_r(0)}(x)-1_{B_r(0)}(y)=0$ (recall that $B_r(0)\subseteq B_R(0)$). 
In both cases, the last term takes the value $2Hh$.
Therefore,
\begin{equation}\label{eq:two}
\msF_s^0(\wphi)=H^2 \msF_s^0(1_{B_R(0)})+h^2 \msF_s^0(1_{B_r(0)})+{\mathscr R},
\end{equation}
where
$\displaystyle
{\mathscr R}=4Hhs\int_{B_r(0)}\,dx\int_{\re^\dd\setminus B_R(0)}\frac1{|x-y|^{\dd+2s}}\, dy$. 

We now argue as in the previous step:
$$
{\mathscr R}=4Hhs\int_{B_r(0)}\left(\int_{\re^\dd\setminus B_{R-|x|}(x)}\frac1{|x-y|^{\dd+2s}}\, dy 
        -\tilde\sigma(x)\right)\,dx,
$$
where
$\displaystyle
\tilde\sigma(x)=\int_{B_R(0)\setminus B_{R-|x|}(x)}\frac1{|x-y|^{\dd+2s}}\, dy$.
Since $B_R(0)\subseteq B_{R+|x|}(x)$, as in Step 1 we have
$$
\tilde\sigma(x)\le \int_{B_{R+|x|}(x)\setminus B_{R-|x|}(x)}\frac{dy}{|x-y|^{\dd+2s}}
    =\frac{\dd\omega_\dd}{2s}\Big(\frac1{(R-|x|)^{2s}}-\frac1{(R+|x|)^{2s}}\Big)\,,
$$
and 
$ \displaystyle
s\int_{B_r(0)}\tilde\sigma(x)\,dx$ is infinitesimal as $s\to0$.

Thus,
$$
{\mathscr R}=4Hhs\int_{B_r(0)}\,dx\int_{\re^\dd\setminus B_{R-|x|}(x)}\frac1{|x-y|^{\dd+2s}}\, dy + o_s(1).
$$
As in the previous step, the inner integral is equal to $\frac{\dd\omega_\dd}{2s}(R-|x|)^{-2s}$. Hence,
\begin{align*}
{\mathscr R}&=2Hh \dd\omega_\dd \int_{B_r(0)}\frac1{(R-|x|)^{2s}}\,dx\cr
    &=2Hh (\dd\omega_\dd)^2 \int_0^r\frac1{(R-\varrho)^{2s}}\varrho^{\dd-1}\,d\varrho
            \underset{s\to0}\to 
            2Hh \dd\omega_\dd|B_r(0)|.
\end{align*}
\par
Recall now \eqref{eq:two} and apply the result of Step 1; then
\begin{align*}
\lim_{s\to0}\msF_s^0(\wphi)=\dd\omega_\dd (H^2 |B_R(0)| + h^2 |B_r(0)| + 2Hh |B_r(0)|)=\dd\omega_\dd\|\wphi\|^2_{L^2(\re^\dd)}\,. \end{align*}
This concludes the proof of Step 2.

\subparagraph{Step 3} 
Let $(r_i)_{i=0,\ldots,n}$ and $(h_i)_{i=0,\ldots,n}$ be strictly positive values with
$R_0 = r_0 > r_1 > \ldots > r_n > 0$.
Let $\wphi=\sum_{i=0}^n h_i 1_{B_{r_i}(0)}$. Then $\wphi$ satisfies \eqref{eq:ptw_conv}.

We first note that we have
\begin{align*}
|\wphi(x)-\wphi(y)|^2&=\Bigl|\sum_i h_i\bigl(1_{B_{r_i}(0)}(x)-1_{B_{r_i}(0)}(y)\bigr)\Bigr|^2\cr
    &=\sum_i h_i^2|1_{B_{r_i}(0)}(x)-1_{B_{r_i}(0)}(y)|^2 \cr
            &\quad+2\sum_{i<j}h_i h_j\bigl(1_{B_{r_i}(0)}(x)-1_{B_{r_i}(0)}(y)\bigr)\bigl(1_{B_{r_j}(0)}(x)-1_{B_{r_j}(0)}(y)\bigr).
\end{align*}
If we use this equality in the expression of $\msF_s^0(w)$,
the first sum on the right-hand side equals $\sum_i h_i^2 \msF_s^0(1_{B_{r_i}(0)})$, which, by Step 1, tends to 
$\dd\omega_\dd\sum_i h_i^2|B_{r_i}(0)|$ as $s\to0$.
For the sum $\sum_{i<j}$, we can argue as in the previous step. Finally, we conclude that
$$
\lim_{s\to0}\msF_s^0(\wphi)=\dd\omega_\dd \Bigl(\sum_i h_i^2|B_{r_i}(0)| + 2\sum_{i<j}h_i h_j|B_{r_j}(0)|\Bigr)\,.
$$
We show that this limit equals the square of the $L^2$ norm of $\wphi$. 
Indeed, since $B_{r_j}(0)\subseteq B_{r_i}(0)$ if $i<j$, we have
\begin{align*}
\int_{\re^\dd}\wphi^2(x)\,dx&=\sum_{i,j}h_i h_j \int_{\re^\dd}1_{B_{r_i}(0)}(x) 1_{B_{r_j}(0)}(x)\, dx\cr
    &=\sum_i h_i^2|B_{r_i}(0)| + 2\sum_{i<j}h_i h_j |B_{r_j}(0)|. 
\end{align*}
This proves Step 3. 
\medskip

We next derive an estimate, which will be needed in the next step, for 
$$
\delta:=\frac12\min_{0\le i\le n-1}(r_i-r_{i+1})
$$ 
in terms of $h$.
Fix a closed subinterval $J$ of $[0,R_0]$ where $\phi'$ does not vanish: let $|\phi'|\ge a>0$
in $J$, for some $a$. Let $h$ be less than half the width of the interval $\phi(J)$; thus, there exists $j$ such that
$jh, (j+1)h\in \phi(J)$. We deduce that
$$
h=\phi(r_{j+1})-\phi (r_j)\ge a (r_j-r_{j+1})\ge 2 a\delta.
$$
On the other hand, if $L$ is a Lipschitz constant of $\phi$, for every $i$ we have
$$
h=\phi(r_{i+1})-\phi(r_i)\le L (r_i-r_{i+1}),
$$
hence $h\le 2L\delta$. We conclude that there exist constants $L,a>0$, depending only on $w$, such that
\begin{equation}\label{eq:est_dh}
 \frac{h}{2L}\le \delta\le \frac{h}{2 a}
\end{equation}
 for $h$ sufficiently small.

\subparagraph{Step 4} {We now fix $w$ as in the statement of the proposition; let $w^h$ be defined as in \eqref{eq:wh}.} We show that 
\begin{equation}\label{eq:conv_unif}
 \lim_{h\to0}\msF_s^0(w^h)=\msF_s^0(w) \qquad\text{uniformly for $s$ sufficiently close to $0$.}
\end{equation}

We first prove an estimate for $|w^h-w|_{2,s}^2$.
Note that the function $\nu_h:=w^h-w$ vanishes outside $B_{R_0}(0)$; then, in the integral that defines $|\nu_h|_{2,s}$,
we require that at least one of the points $x$ and $y$ lies
in $B_{R_0}(0)$:
$$
|\nu_h|_{2,s}^2\le 2\int_{B_{R_0}(0)} \int_{\re^\dd}\frac{|\nu_h(x)-\nu_h(y)|^2}{|x-y|^{\dd+2s}}\,dy\, dx.  
$$
\par
Let now $\delta=\frac12\min_i(r_i-r_{i+1})$, 
and divide the inner $y$-integral into two sets,
depending on $x\in B_{R_0(0)}$: 
\begin{align*}
 |\nu_h|_{2,s}^2=2\int_{B_{R_0}(0)} dx&\int_{\re^\dd\setminus B_\delta(x)}\frac{|\nu_h(x)-\nu_h(y)|^2}{|x-y|^{\dd+2s}}\,dy\\
    &+ 2\int_{B_{R_0}(0)} dx\int_{B_\delta(x)}\frac{|\nu_h(x)-\nu_h(y)|^2}{|x-y|^{\dd+2s}}\,dy
    =:I_1+I_2.
\end{align*}
\par
To estimate $I_1$ we note that $|w^h-w|\le h$, hence $|\nu_h(x)-\nu_h(y)|\le 2h$. Then
\begin{equation*}
I_1\le 2(2h)^2\int_{B_{R_0}(0)} dx\int_{\re^\dd\setminus B_\delta(x)}\frac1{|x-y|^{\dd+2s}}\,dy
    =8\dd\omega_\dd h^2 \delta^{-2s}\frac{|B_{R_0}(0)|}{2s}\,.
\end{equation*}
\par
We now address $I_2$, for which we use the following decomposition:
$$
I_2=2\sum_{i=0}^n\int_{S_i} dx \int_{B_\delta(x)}\frac{|\nu_h(x)-\nu_h(y)|^2}{|x-y|^{\dd+2s}}\,dy, 
$$
where $S_i= B_i\setminus \overline B_{i+1}$ if $i<n$, and $S_n=B_n$.
\par
Fix $i<n$. For every $x\in S_i$ and $y\in B_\delta(x)$ we have:
\begin{equation}\label{eq:w_split}
|\nu_h(x)-\nu_h(y)|^2\le 2\bigl(|w^h(x)-w^h(y)\bigr|^2+|w(x)-w(y)|^2\bigr).
\end{equation}
We have $w^h(x)-w^h(y)=0$ unless
$d(x,\partial S_i)<\delta$, and $y\in B_\delta(x)\setminus S_i$.
In this case, since $\delta$ is at most half the width of every $S_j$,
we have $|w^h(x)-w^h(y)|=h$. Then
\begin{equation}\label{eq:bdry_Si}
\begin{split}
\int_{S_i} dx\int_{B_\delta(x)}&\frac{|w^h(x)-w^h(y)|^2}{|x-y|^{\dd+2s}}\,dy\\
    &=\int_{\{x\in S_i:\ d(x,\partial S_i)<\delta\}} dx\int_{B_\delta(x)\setminus S_i}\frac{h^2}{|x-y|^{\dd+2s}}\,dy\,.
\end{split}
\end{equation}

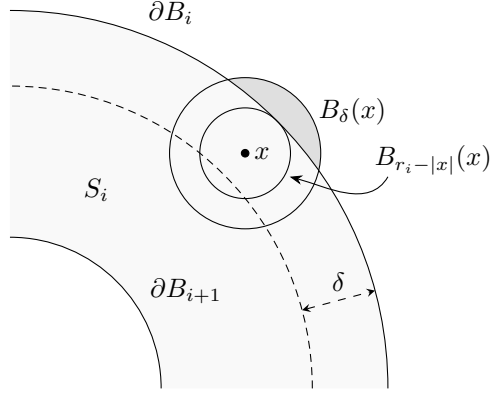
\begin{figure}[h!]
\begin{center}
\begin{tikzpicture}
 \tikzmath{\c0=5; \c1=0.88;}
 \path[fill=gray!25] ({\c1*\c0/sqrt(2)},{\c1*\c0/sqrt(2)}) circle [radius=1];
 \path[fill=gray!5] (\c0,0) arc (0:90:\c0)--(0,.4*\c0) arc (90:0:.4*\c0)--cycle;
 \draw ({\c1*\c0/sqrt(2)},{\c1*\c0/sqrt(2)}) circle [radius=1];
 \draw (\c0,0) arc (0:90:\c0);
 \draw[densely dashed] (\c0-1,0) arc (0:90:\c0-1);
 \draw (.4*\c0,0) arc (0:90:.4*\c0);
 \draw[fill] ({\c1*\c0/sqrt(2)},{\c1*\c0/sqrt(2)}) circle [radius=0.05] node[right] {$x$};
 \draw ({\c1*\c0/sqrt(2)},{\c1*\c0/sqrt(2)}) circle [radius=(1-\c1)*\c0];
 \draw[dashed,<->,>=stealth] (15:\c0-1)--(15:\c0);
 \node[above] at (15:\c0-.5) {$\delta$};
 \node[below right] at ({(\c1*\c0+1.2)/sqrt(2)},{(\c1*\c0+1.2)/sqrt(2)}) {$B_\delta(x)$};
 \node[above right] at (30:.4*\c0) {$\partial B_{i+1}$};
 \node[above right] at (70:\c0) {$\partial B_i$};
 \node[above right] at (70:.5*\c0) {$S_i$};
\draw[-{Stealth[scale length=1]}] (5,2.8) .. controls +(-120:0.5cm) and +(-15:0.5cm) .. (3.7,2.8);
\node[above] at (5.6,2.7) {$B_{r_i-|x|}(x)$};
\end{tikzpicture}
\caption{when $x$ is $\delta$-close to $\partial S_i$, the darker shaded part of the ball $B_\delta(x)$ (where the difference $w^h(x)-w^h(y)$ does not vanish) is controlled by the annular region.}
\label{fig:B}
\end{center}
\end{figure}
Assume now, for example, that $r_i-\delta<|x|<r_i$
(see Figure \ref{fig:B}), which implies that $y\in S_{i-1}$ (the case $r_{i+1}<|x|<r_{i+1}+\delta$
being analogous). Then
\begin{align*}
\int_{B_\delta(x)\setminus S_i}\frac1{|x-y|^{\dd+2s}}\,dy
    &\le \int_{\{y:\ r_i-|x|<|y-x|<\delta\}}\frac1{|x-y|^{\dd+2s}}\,dy\\
    &=\frac{\dd\omega_\dd }{2s}\bigl[(r_i-|x|)^{-2s}-\delta^{-2s}\bigr].
\end{align*}
We now integrate the term in square brackets with respect to $x$ in the outer strip $r_i-\delta<|x|<r_i$:
\begin{align*}
 &\hspace{-2em}\int_{\{x:\ r_i-\delta<|x|<r_i\}} \bigl[(r_i-|x|)^{-2s}-\delta^{-2s}\bigr]\,dx\\
    &=\dd\omega_\dd\int_{r_i-\delta}^{r_i}\bigl[(r_i-\varrho)^{-2s}-\delta^{-2s}\bigr]\varrho^{\dd-1}\,d\varrho\\
    &\le \dd\omega_\dd r_i^{\dd-1}\int_{r_i-\delta}^{r_i}\bigl[(r_i-\varrho)^{-2s}-\delta^{-2s}\bigr]\,d\varrho
    =\frac{2s\dd\omega_\dd}{1-2s}r_i^{\dd-1}\delta^{1-2s}.
\end{align*}
An analogous estimate holds for the inner strip.
\par
Returning to \eqref{eq:bdry_Si}, we get the following:
$$
\int_{S_i} dx\int_{B_\delta(x)}\frac{|w^h(x)-w^h(y)|^2}{|x-y|^{\dd+2s}}\,dy
    \le 2 \frac{(\dd\omega_\dd)^2}{1-2s} R_0^{\dd-1} h^2\delta^{1-2s}.
$$
\par
We now consider the integral of the term $|w(x)-w(y)|^2$ in \eqref{eq:w_split}.
Let $L$ be a Lipschitz constant for $w$. Then
\begin{align*}
\int_{S_i}\ dx\int_{B_\delta(x)}\frac{|w(x)-w(y)|^2}{|x-y|^{\dd+2s}}\,dy 
    &\le L^2 \int_{S_i} dx\int_{B_\delta(x)}\frac1{|x-y|^{\dd-2+2s}}\,dy\\
    &=L^2|S_i|\frac{\dd\omega_\dd}{2-2s} \delta^{2-2s}.
\end{align*}
We conclude that for $s$ sufficiently small (say $s<1/3$) and $i=0,\ldots,n-1$:
$$
\int_{S_i} dx\int_{B_\delta(x)}\frac{|\nu_h(x)-\nu_h(y)|^2}{|x-y|^{\dd+2s}}\,dy
    \le C(N,w)\bigl(h^2\delta^{1-2s}+ |S_i|\delta^{2-2s}\bigr),
$$
where $C(N,w)$ is a constant depending only on $\dd$ and $w$. The same inequality holds for $S_n$ (only the outer boundary has to be considered). Finally, we estimate $I_2$ using the fact that $n\le M/h$,
so that we can assume $(n+1)h\le 2M$:
$$
I_2\le 2C(N,w)\bigl(2M h\delta^{1-2s} + |B_{R_0}(0)| \delta^{2-2s}\bigr).
$$
\par
Now, we gather the estimate for $I_1$ and $I_2$, and take into account the inequalities \eqref{eq:est_dh}. Then
\begin{equation}\label{eq:est_nu}
|\nu_h|_{2,s}^2\le 2\,C(w) h^{2-2s},
\end{equation}
where $C(w)$ denotes a constant that depends only on $w$. 
\par

Since $\msF_s^0(f)^{1/2}=\sqrt s\,|f|_{2,s}$ is a seminorm,
\[
\bigl|\msF_s^0(w^h)^{1/2}-\msF_s^0(w)^{1/2}\bigr|\le \msF_s^0(\nu_h)^{1/2}\le C h^{1-s},
\]
and $\msF_s^0(w)$ is bounded for small $s$ by Theorem~\ref{thm:pws}. Therefore
$|\msF_s^0(w^h)-\msF_s^0(w)|\le C h^{1-s}\le Ch^{1-s_0}$ for $s\le s_0$.
This concludes Step 4, and the proof of the proposition. \qed

\section{Reduction to pointwise convergence via rearrangements}\label{sec:rearr}

The key idea of the proof of Theorem \ref{thm:main} is to replace functions $u_s$ by their spherically symmetric rearrangements. This is possible since the energy decreases (in the limit) by rearrangements (Step 1 in this section). Symmetric functions enjoy better convergence properties and allow passing from a weakly convergent $u_s$ to a strongly convergent sequence (Step 2). In turn, the strong convergence makes it possible to replace the sequence with its limit (Step 3), thus reducing the problem to a pointwise convergence of the energy, which is proved in the previous section.

We note two technical points. First, the property of strong convergence of rearrangements demands an equiboundedness assumption, which will be subsequently removed (Step 4). Second, {since the reduction to a pointwise-convergence problem can be carried out more easily for smooth functions, we make use of a standard regularizing procedure in Step 3.}

As we will see, the case $r_s\equiv+\infty$ makes the structure of the proof more transparent, and it represents an essential part of the analysis. In the following steps 
we now show how the general case reduces to this one  
by relying on the decomposition in \eqref{eq:dec_F}.

\smallskip
Consider $u_s\in L^2(\re^\dd)$, with $s>0$, and such that
$$
\hbox{$u_s\rightharpoonup u$ weakly in $L^2(\mathbb R^N)$ as $s\to0$.}
$$
We assume that there exists $\Rbar>0$ such that  
the functions $u_s$ have support in a common compact subset of the open ball $B_{R_0}(0)$.
Moreover, up to a subsequence we can suppose that $\liminf$ in the statement of Theorem~\ref{thm:main} is a limit, so that it does not change by considering further subsequences.

As mentioned above, we first argue under the additional assumption that 
\begin{equation}\label{eq:bound}
\text{there exists $M>0$ such that $|u_s|\le M$.}
\end{equation}
\par
If $w\colon\re^\dd\to\re$ is a Lebesgue-measurable function, we denote by $\sphr{w}$ the \emph{spherically symmetric rearrangement} of $w$ \cite[Definitions 10.16 and 10.33]{Leoni:frac}.
If $w$ vanishes at infinity, then (see \cite[Theorem 10.34]{Leoni:frac}) 
$
|\sphr{w}|_{2,s}\le |w|_{2,s}.
$
This inequality yields $\msF_s^0(\sphr{u}_s)\le \msF_s^0(u_s)$. The next step shows that an analogous inequality holds
for $\msF_s$ as $s\to0$.

\subparagraph{Step 1} We prove that for every $s>0$
\begin{equation}\label{eq:lwe}
\msF_s(\sphr{u}_s)\le \msF_s(u_s)+o_s(1). 
\end{equation}

\par
Apply \eqref{eq:dec_F} to  $\msF_s(u_s)$ and $\msF_s(\sphr{u}_s)$; we have
$$
\msF_s(\sphr{u}_s)-\msF_s(u_s)=\msF_s^0(\sphr{u}_s)-\msF_s^0(u_s)+\sigma(s),
$$
where 
$$
\sigma(s)=\delta_s\bigl(u_s^2-(\sphr{u}_s)^2\bigr)
    =\int_{B_{R_0}}\eta_s(x)\bigl(u_s^2(x)-(\sphr{u}_s)^2(x)\bigr)\,dx.
$$
Since $\msF_s^0(\sphr{u}_s)\le \msF_s^0(u_s)$, it is enough to show that $\sigma(s)\to0$ as $s\to0$.
Recalling that $u_s$ and $\sphr{u}_s$ have the same $L^2$ norm, we can write
$$
\sigma(s)=\int_{B_{R_0}} (\eta_s(x)-\lambda\dd\omega_\dd)\bigl(u_s^2(x)-(\sphr{u}_s)^2(x)\bigr)\,dx\,.
$$
We now get \eqref{eq:lwe} by the $L^2$ equiboundedness of $u_s$ and $\sphr{u}_s$ and by Proposition~\ref{prop:corr}.

\subparagraph{Step 2} 
Since $\sphr{u}_s$ is radially symmetric and radially decreasing
(see \cite[Theorem 10.17]{Leoni:frac}) there exists a decreasing function
$\f_s\colon[0,+\infty)\to[0,+\infty)$ such that
$$
\sphr{u}_s(x)=\f_s(|x|),\qquad x\in\re^\dd.
$$
From \eqref{eq:bound} we also have
$$
0\le \f_s \le M;\qquad
\f_s(r)=0\quad\text{if $r>\Rbar$.}
$$
By Helly's Theorem there exists a subsequence of $(\f_s)_s$
that converges pointwise in $[0,\Rbar]$.
We can assume, for simplicity, that the whole sequence $(\f_s)_s$ converges; if we denote by $\f$ the limit, we deduce that

\begin{equation}\label{eq:strongconv}
\sphr{u}_s\to v\quad\text{strongly in $L^2(\re^\dd)$, where $v(x)=\varphi(|x|)$.}
\end{equation}

It remains to prove Theorem \ref{thm:main} for the sequence 
$(\sphr{u}_s)_s$, which strongly converges to $v$; i.e., we have to prove that
\begin{equation}\label{eq:non_smooth}
 (1-\lambda)\dd \omega_\dd \|v\|_{L^2}^2\le \liminf_{s\to0} \msF_s(\sphr{u}_s).
\end{equation}
Indeed, $\msF_s(\sphr{u}_s)\le \msF_s(u_s)+o_s(1)$ by \eqref{eq:lwe}, and
$\|v\|_{L^2(\re^\dd)}\ge \|u\|_{L^2(\re^\dd)}$. This last inequality follows from
 $\|\sphr{u}_s\|_{L^2(\re^\dd)}=\|u_s\|_{L^2(\re^\dd)}$, since
\begin{equation}\label{eq:conv_debole}
\|v\|_{L^2(\re^\dd)}=\lim_{s\to0}\|\sphr{u}_s\|_{L^2(\re^\dd)}
    =\lim_{s\to0}\|u_s\|_{L^2(\re^\dd)}\ge \|u\|_{L^2(\re^\dd)}\,.
\end{equation}
\begin{remark}\label{rk:weak_con}\normalfont
 For future reference we explicitly note that the last inequality in \eqref{eq:conv_debole} is the only point where the weak convergence of $(u_s)$ is used. In particular, this will be used in the final step of this section, when we will remove the additional assumption \eqref{eq:bound}.
\end{remark}
Thanks to the previous steps we have reduced the problem to the proof of~\eqref{eq:non_smooth}.

\subparagraph{Step 3} ({\it Reduction to pointwise convergence; regularization})
In order to prove \eqref{eq:non_smooth}:
\begin{itemize}
 \item[$a)$] we replace $\liminf_{s\to 0} \msF_s(\sphr{u}_s)$ with $\liminf_{s\to0} \msF_s(v)$, by exploiting the strong 
 $L^2$ convergence of $(\sphr{u}_s)$ to $v$;
 \item[$b)$] we apply the pointwise convergence of $\msF_s$ for spherically symmetric and radially decreasing functions obtained in the previous section.
\end{itemize}
{It is convenient to first prove claim (a) for the regularized
functions; we will then obtain the general case by letting the
smoothing parameter $\e$ tend to $0$.}

Let $(\varrho_\e)_{\e}$ be standard symmetric regularizing kernels. Define
\begin{equation}\label{eq:reg}
\overline u^\e_s=\varrho_\e*\sphr{u}_s,\qquad v^\e=\varrho_\e*v.
\end{equation}
For every $w\in L^2_{\text{loc}}(\re^\dd)$ with $|w|_{2,s}<+\infty$,
we know that
$
|\varrho_\e*w|_{2,s}\le |w|_{2,s}
$   (\cite{Leoni:frac}, Theorem 6.62). 
This implies that $\msF_s^0(\overline u^\e_s)\le \msF_s^0(\sphr{u}_s)$. We next extend this inequality to $\msF_s$. More precisely, we claim that

\begin{equation}\label{prop:reg} \liminf_{s\to0}\msF_s(\overline u_s^\e)\le \liminf_{s\to0}\msF_s(\sphr{u}_s)+ o_\e(1).
\end{equation}
To prove this claim, we
apply \eqref{eq:dec_F} to $\msF_s(\overline u_s^\e)$ and $\msF_s(\sphr{u}_s)$, and use the inequality $\msF_s^0(\overline u^\e_s)\le \msF_s^0(\sphr{u}_s)$. Then we have
$$
\msF_s(\overline u_s^\e)-\msF_s(\sphr{u}_s)=\msF_s^0(\overline u^\e_s)- \msF_s^0(\sphr{u}_s)+\sigma^\e(s)\le \sigma^\e(s),
$$
where $\sigma^\e(s)=\delta_s\bigl((\sphr{u}_s)^2-(\overline u_s^\e)^2\bigr)$,
so that
\begin{align*}
|\sigma^\e(s)|&\le\int_{B_{R_0}}\eta_s \bigl|(\sphr{u}_s)^2-(\overline u_s^\e)^2\bigr|\,dx\\
        &\le \int_{B_{R_0}}|\eta_s-\lambda\dd\omega_\dd| \bigl|(\sphr{u}_s)^2-(\overline u_s^\e)^2\bigr| dx 
            + \lambda\dd\omega_\dd\int_{B_{R_0}}\bigl|(\sphr{u}_s)^2-(\overline u_s^\e)^2\bigr| dx.
\end{align*}
By Proposition \ref{prop:corr} and the $L^2$ boundedness of $u_s$ (and, hence, of $\sphr{u}_s$ and $\overline u_s^\e$),
the first integral in the right-hand side of the last inequality tends to zero as $s\to0$
(uniformly with respect to $\e$). The same $L^2$ boundedness applied to the other integral yields the following estimate
($C$ a positive constant independent of $s$ and $\e$):
\begin{align*}
\int_{B_{R_0}}|\sphr{u}_s+\overline u_s^\e| |\sphr{u}_s-\overline u_s^\e| dx
    &\le C\|\sphr{u}_s-\overline u_s^\e\|_{L^2(\re^\dd)}\\
    &\le C(\|\sphr{u}_s-v\|_{L^2}+\|v-v^\e\|_{L^2}+\|v^\e-\overline u_s^\e\|_{L^2})\\
    &\le C(2\|\sphr{u}_s-v\|_{L^2}+\|v-v^\e\|_{L^2}),
\end{align*}
since $\|v^\e-\overline u_s^\e\|_{L^2}=\|(v-\sphr{u}_s)^\e\|_{L^2}\le \|v-\sphr{u}_s\|_{L^2}$. Therefore, we get
$$
\limsup_{s\to0}|\sigma^\e(s)|\le C\|v-v^\e\|_{L^2(\re^\dd)}\,,
$$
concluding the proof. This proves claim \eqref{prop:reg}.

\medskip
\par
We now prove (a) with $\overline u^\e_s$ and $v^\e$ in place of $\sphr{u}_s$ and $v$, respectively. To this end we claim that
\begin{equation}\label{prop:linf_lim}
\liminf_{s\to0} \msF_s(\overline u^\e_s)=\liminf_{s\to0} \msF_s(v^\e).
\end{equation}

We first address the case $r_s\equiv+\infty$.
Let $\e>0$ be fixed. Since the functions $\sphr{u}_s-v$ are equibounded in $L^1(\re^\dd)$, we can choose a Lipschitz constant $L_\e$ for $\overline u^\e_s-v^\e=(\sphr{u}_s-v)^\e$ which is independent of $s$: indeed, a Lipschitz constant is given by $\|\nabla\varrho_\e\|_\infty\,\|\sphr{u}_s-v\|_{L^1}$.
We now show that
\begin{equation}\label{eq:convregol}
\lim_{s\to0} s|\overline u^\e_s-v^\e|^2_{2,s}=0.
\end{equation}
For every $R>1$ we have that
\begin{align*}
 I:&=|\overline u^\e_s-v^\e|^2_{2,s}
    =\int_{\re^\dd}\frac1{|\xi|^{\dd+2s}}\, d\xi\, \int_{\re^\dd}|\Delta_\xi(\overline u^\e_s-v^\e)|^2(x)\,dx\\
    &=\left(\int_{B_R(0)}\frac1{|\xi|^{\dd+2s}}\, d\xi + \int_{\re^\dd\setminus B_R(0)}\frac1{|\xi|^{\dd+2s}}\, d\xi\right)
            \int_{\re^\dd}|\Delta_\xi(\overline u^\e_s-v^\e)(x)|^2\,dx\\
    &=: I_1 + I_2,
\end{align*}
where we have set $(\Delta_\xi f)(x)=f(x+\xi)-f(x)$.
We now estimate $I_1$. Take into account that the support of $(\overline u^\e_s-v^\e)$ is contained in $B_{R_0}(0)$, and let 
$c=\lfloor R_0/R\rfloor+2$; then $R+R_0\le cR$ and
$$
I_1\le \int_{B_R(0)}\frac1{|\xi|^{\dd+2s}} (L_\e |\xi|)^2 |B_{cR}(0)|\, d\xi
    =C R^\dd\int_0^R \frac1{\varrho^{\dd+2s-2}}\varrho^{\dd-1}\,d\varrho,
$$
where $C$ is a positive constant independent of $s$. We conclude that, for $s<1/2$,
\begin{equation}\label{eq:est:I1}
I_1\le C R^{\dd+2-2s}.
\end{equation}
\par
We now address $I_2$. Note that for every $\xi\in\re^\dd$
\begin{align*}
\int_{\re^\dd}|\Delta_\xi(\overline u^\e_s-v^\e)(x)|^2\,dx &
    \le 2\left(\int_{\re^\dd}|(\overline u^\e_s-v^\e)(x+\xi)|^2\,dx +
    \int_{\re^\dd}|(\overline u^\e_s-v^\e)(x)|^2\,dx\right)\cr
    &=4\|\overline u^\e_s-v^\e\|_{L^2(\re^\dd)}^2\le 4\|\sphr{u}_s-v\|_{L^2(\re^\dd)}^2\,.
\end{align*}
Therefore
\begin{align*}
I_2&=\int_{\{|\xi|\ge R\}} \frac1{|\xi|^{\dd+2s}} d\xi\int_{\re^\dd}|\Delta_\xi(\overline u^\e_s-v^\e)(x)|^2\,dx\cr
    &\le 4\|\sphr{u}_s-v\|_{L^2(\re^\dd)}^2\int_{\{|\xi|\ge R\}} \frac1{|\xi|^{\dd+2s}}\, d\xi
    \le 2\dd\omega_\dd\|\sphr{u}_s-v\|_{L^2(\re^\dd)}^2\, \frac1{s R^{2s}}.
\end{align*}
\par
Combining the two estimates and recalling that $R>1$, we have
\begin{align*}
s|\overline u^\e_s-v^\e|^2_{2,s}&= sI\le C\bigl(sR^{\dd+2-2s} + \|\sphr{u}_s-v\|_{L^2(\re^\dd)}^2\, R^{-2s}\bigr)\cr
    &\le sCR^{\dd+2-2s} + 2\|\sphr{u}_s-v\|_{L^2(\re^\dd)}^2.
\end{align*}
Passing to the limit as $s\to0$ and using \eqref{eq:strongconv}, we obtain \eqref{eq:convregol}.
\par
The seminorm property $\bigl||\overline u^\e_s|_{2,s}-|v^\e|_{2,s}\bigr|\le |\overline u^\e_s-v^\e|_{2,s}$, 
together with \eqref{eq:convregol}, now implies that
$$
\lim_{s\to0}|\bigl(\msF_s^0(\overline u_s^\e)\bigr)^{1/2} - \bigl(\msF_s^0(v^\e)\bigr)^{1/2}|=0,
$$
from which $\liminf_{s\to0}\bigl(\msF_s^0(\overline u_s^\e)\bigr)^{1/2}=\liminf_{s\to0} \bigl(\msF_s^0(v^\e)\bigr)^{1/2}$. Claim \eqref{prop:linf_lim} follows in the case $r_s\equiv+\infty$.
\par
We now consider the general case. Apply \eqref{eq:dec_F} to $\msF_s(\overline u^\e_s)$ and 
$\msF_s(v^\e)$, and consider $\delta_s\bigl((\overline u^\e_s)^2 - (v^\e)^2\bigr)$.
By the $L^2$ boundedness of $u_s$ there exists $C>0$ such that
$$
\int_{B_{R_0}}\bigl|(\overline u^\e_s)^2 - (v^\e)^2\bigr|\,dx\le C\|\overline u^\e_s - v^\e\|_{L^2(\re^\dd)}
    \le C\|\sphr{u}_s - v\|_{L^2(\re^\dd)}\to 0\quad\text{as $s\to0$.}
$$
Together with the uniform convergence of $\eta_s$ (Proposition \ref{prop:corr}) this implies that
$\delta_s\bigl((\overline u^\e_s)^2 - (v^\e)^2\bigr)\to0$ as $s\to0$. Since $\delta_s\bigl((v^\e)^2\bigr)$ 
has a limit (with value $\lambda\dd\omega_\dd\|v^\e\|_{L^2}^2$) as $s\to0$, then
$\delta_s\bigl((\overline u^\e_s)^2\bigr)$ has limit, too, with the same value. Hence, taking into account what already proved above for $\msF_s^0$, we deduce that
\begin{align*}
 \liminf_{s\to0}\msF_s(\overline u_s^\e)&=\liminf_{s\to0}\bigl[\msF_s^0(\overline u_s^\e)
    -\delta_s\bigl((\overline u^\e_s)^2\bigr)\bigr]\\
    &=\liminf_{s\to0}\bigl[\msF_s^0(\overline u_s^\e)
    -\delta_s\bigl((v^\e)^2\bigr)\bigr]\\
    &=\liminf_{s\to0}\bigl[\msF_s^0(v^\e)
    -\delta_s\bigl((v^\e)^2\bigr)\bigr]=\liminf_{s\to0}\msF_s(v^\e);
\end{align*}
that is, \eqref{prop:linf_lim}.

\subparagraph{Step 4} 
We now collect the results above, and take the pointwise convergence proved in the previous section into account. In this way, we have proved Theorem \ref{thm:main} under the additional assumption 
\eqref{eq:bound}. Next, we will remove this condition.
\par
By \eqref{eq:lwe} (Step 1) we have
$$
\liminf_{s\to0}\msF_s(u_s)
\ge\liminf_{s\to0}\msF_s(\sphr{u}_s), 
$$
and by \eqref{prop:reg} and \eqref{prop:linf_lim} (Step 3)
\begin{equation}\label{eq:collect}
\liminf_{s\to0} \msF_s(\sphr{u}_s)+ o_\e(1) \ge \liminf_{s\to0}\msF_s(\overline u_s^\e)
=\liminf_{s\to0}\msF_s(v^\e). 
\end{equation}
The function $v^\e$ defined in \eqref{eq:reg} is radially symmetric, radially decreasing and has compact support.
Thus, $v^\e(x)=\phi^\e(|x|)$, with $\phi^\e$ non-negative, decreasing $C^1$ function on $[0,+\infty)$ with compact support.
Therefore, we can apply Proposition~\ref{prop:pws_radial} with $\wphi=v^\e$ and $\phi=\phi^\e$,
and obtain that the liminf on the right-hand side equals 
$(1-\lambda)\dd\omega_\dd\|v^\e\|_{L^2(\re^\dd)}^2$. 
Finally, letting $\e$ go to zero 
and applying \eqref{eq:conv_debole} (Step 2) we conclude.
\medskip

We finally remove assumption \eqref{eq:bound}. For every $M>0$ and $w\colon \re^\dd\to\re$ let $w^M=(w\wedge M)\vee(-M)$.
Since $t\mapsto (t\wedge M)\vee(-M)$ is a Lipschitz function with Lipschitz constant 1, we have
$$
\msF_s(u_s^M)\le \msF_s(u_s).
$$
Taking Remark \ref{rk:weak_con} into account, we obtain that
$$
\liminf_{s\to0} \msF_s(u_s^M)\ge (1-\lambda)\dd\omega_\dd\|v_M\|_{L^2(\re^\dd)}^2\,,
$$
where $v_M$ is the strong $L^2$-limit of $\sphr{(u_s^M)}$ (recall \eqref{eq:strongconv}).
Since (compare \eqref{eq:conv_debole})
$$
\|v_M\|_{L^2(\re^\dd)}=\lim_{s\to0}\|\sphr{(u_s^M)}\|_{L^2(\re^\dd)}
    =\lim_{s\to0}\|u_s^M\|_{L^2(\re^\dd)}\,,
$$
 we conclude by choosing the truncation level $M$ in a suitable way. This follows by a general result on weakly converging sequences, which we state and prove in the following Lemma~\ref{lem}. \qed

\begin{lemma}\label{lem}
 Let $(w_n)_n$ be a weakly convergent sequence in $L^2(\re^\dd)$. Let $w$ be the limit function.
 Then, for every $\e>0$ there exists $M>0$ such that 
 $$
 \liminf_{n\to+\infty}\|w_n^M\|_{L^2(\re^\dd)}\ge \|w\|_{L^2(\re^\dd)}-\e,
 $$
where $w_n^M=(w_n\wedge M)\vee(-M)$.
\end{lemma}

\proof  We may assume that $w\not\equiv 0$. For every $M>0$ we have
$$
 \int_{\re^\dd} w_n^M w\, dx
    =\int_{\re^\dd} w_n w\,dx - \int_{\re^\dd} (w_n- w_n^M)w\,dx\,.
$$
We show that the second integral on the right-hand side can be made arbitrarily small for a suitable $M$ independent of $n$.
Indeed,
\begin{align*}
\left| \int_ {\re^\dd} (w_n-w_n^M) w\, dx\right|&\le \int_{\{|w_n|>M\}}(|w_n| -M) |w|\, dx\le \int_{\{|w_n|>M\}}|w_n w|\, dx\\
    &\le \| w_n\|_{L^2(\re^\dd)} \left( \int_{\{|w_n|>M\}}w^2\, dx\right)^{1/2}.
\end{align*}
The weak convergence of $(w_n)$ implies that $\|w_n\|_{L^2(\re^\dd)}\le C$ for a suitable $C$ independent of $n$. 
By Chebyshev's inequality $|\{|w_n|>M\}|\le (C/M)^2$. Thus, the integrability of $w^2$ implies that
for every $\sigma>0$ we can find a value $M$ such that
$\bigl(\int_{\{|w_n|>M\}}w^2\, dx\bigr)^{1/2}\le \sigma$ for every $n$. 
Then
$$
\liminf_{n\to+\infty} \int_{\re^\dd} w_n^M w\, dx\ge \liminf_{n\to+\infty} \int_{\re^\dd} w_n w\,dx - \sigma C
    =\|w\|_{L^2(\re^\dd)}^2- \sigma C.
$$
The left-hand side is majorized by $\|w\|_{L^2}\liminf_n\|w_n^M\|_{L^2}$. Thus
$$
\liminf_n\|w_n^M\|_{L^2}\ge \|w\|_{L^2} -\sigma C/\|w\|_{L^2}.
$$
We conclude by taking $\sigma=\e \|w\|_{L^2}/C$. \qed

\section{Convergence on smooth functions}\label{sec:limsup}

In Remark \ref{rk:dec_F} we noted that it is enough to prove Theorem \ref{thm:pws} 
in the case $r_s\equiv +\infty$; i.e., for $\msF_s^0$.
As we mentioned, this result
is contained in Theorem 3 of \cite{Maz-Shap02}.
For completeness, we adapt the proof here to Lipschitz functions with compact support.

\par
Let
$u$ be a Lipschitz function on $\re^\dd$ with compact support.
In the integral defining the Gagliardo $s$-seminorm of $u$, we split $\re^\dd\times\re^\dd$ into the set $\{(x,y)\in \re^\dd\times\re^\dd:\ |y|\ge 2|x|\ \text{or}\ |x|\ge2|y|\}$ and its complement. By symmetry, $\msF_s^0(u)=s|u|_{2,s}^2=2s(I_1+I_2)$, with
$$
I_1=\int\int_{\{|y|\ge 2|x|\}}\frac{|u(x)-u(y)|^2}{|x-y|^{\dd+2s}}\,dx\, dy\,,
$$
and
$$
I_2=\int\int_{\{|x|<|y|< 2|x|\}}\frac{|u(x)-u(y)|^2}{|x-y|^{\dd+2s}}\,dx\, dy\,.
$$
By the Minkowski inequality (with respect to the measure $1/|x-y|^{\dd+2s}dx\,dy$) we get
$$
I_1^{\frac12}\le \Bigl(\int\int_{\{|y|\ge 2|x|\}}\frac{|u(x)|^2}{|x-y|^{\dd+2s}}\, dx dy\Bigr)^{\frac12}
    + \Bigl(\int\int_{\{|y|\ge 2|x|\}}\frac{|u(y)|^2}{|x-y|^{\dd+2s}}\, dx dy\Bigr)^{\frac12}.
$$
Since $|y|\ge 2|x|$ implies that $|x-y|\ge |x|$, for the first integral on the right-hand side, we have
\begin{align*}
\int_{\re^\dd}|u(x)|^2 dx &\int_{\{|y|\ge 2|x|\}}\frac1{|x-y|^{\dd+2s}}\,dy\\
    &\hspace{-2em}\le \int_{\re^\dd}|u(x)|^2 dx \int_{\{|x-y|\ge |x|\}}\frac1{|x-y|^{\dd+2s}}\,dy
    =\frac{\dd\omega_\dd}{2s}\int_{\re^\dd}\frac{|u(x)|^2}{|x|^{2s}}\,dx.
\end{align*}
For the second integral, we note that the condition $|y|\ge 2|x|$ implies $|x-y|\ge |y|/2$; hence,
\begin{align*}
\int_{\re^\dd}|u(y)|^2 dy &\int_{\{|y|\ge 2|x|\}}\frac1{|x-y|^{\dd+2s}}\,dx\\
    &\hspace{-2em}\le \int_{\re^\dd}|u(y)|^2 dy \int_{\{\{|x|<\frac12|y|\}}\frac1{(|y|/2)^{\dd+2s}}\,dx
    =2^{2s}\omega_\dd\int_{\re^\dd}\frac{|u(y)|^2}{|y|^{2s}}\,dy.
\end{align*}
If we collect the above estimates, we deduce that
$$
sI_1\le \left(\Bigl(\frac{\dd\omega_\dd}{2}
    \int_{\re^\dd}\frac{|u(x)|^2}{|x|^{2s}}\,dx\Bigr)^{1/2} + 
    2^s\Bigl(\omega_\dd s\int_{\re^\dd}\frac{|u(y)|^2}{|y|^{2s}}\,dy\Bigr)^{1/2}\right)^2.
$$
We now note that $\int_{\re^\dd\setminus B_1(0)}|u(x)|^2/|x|^{2s} dx
\le \int_{\re^\dd\setminus B_1(0)}|u(x)|^2 dx$,
while in $B_1(0)$ the function $|u(x)|^2/|x|^{2s}$ is integrable (since $u$ is bounded) 
and decreasing as $s\to0$; then
$$
\limsup_{s\to0}\int_{\re^\dd} \frac{|u(x)|^2}{|x|^{2s}}\, dx\le\|u\|_{L^2(\re^\dd)}^2.
$$
It follows that $\limsup_{s\to0} 2sI_1\le \dd\omega_\dd \|u\|_{L^2(\re^\dd)}^2$. It remains to show that
$s I_2$ vanishes as $s\to0$.
\par
Since $u$ has compact support, the set of points $(x,y)$ for which
the integrand in $I_2$ does not vanish is contained in a bounded subset $E$ of $\re^{2\dd}$.
Hence, if we denote by $L$ a Lipschitz constant for $u$, we have
$$
sI_2\le sL^2\int\int_E\frac1{|x-y|^{\dd-2+2s}} dx\,dy\to0\quad\text{as $s\to0$.} \qed
$$

\bigskip\noindent
{\bf Acknowledgments.} 
 Andrea Braides acknowledges funding by the Excellence Project MatMod@TOV (Grant Agreement No.~CUP\_E83C23000330006) of the
Department of Mathematics of the University of Rome Tor Vergata. 
Anneliese Defranceschi acknowledges financial support from the ``Calcolo delle Variazioni e Teoria Geometrica della Misura'' funds of the Department of Mathematics of the University of Trento.

\end{document}